\documentclass[a4paper]{article}
\usepackage[T2A]{fontenc}
\usepackage[cp1251]{inputenc} 
\usepackage[english]{babel}
\usepackage[tbtags]{amsmath}
\usepackage{amsfonts,amssymb,mathrsfs,amscd,comment}
\usepackage{amsthm}

\newtheorem{defn}{Definition}
\newtheorem{prop}{Proposition}
\newtheorem{thm}{Theorem}

\begin{document}

\title{\bf\large\MakeUppercase{%
Degenerate Fourier transform associated with the Sturm-Liouville operator
}}

\author{A.\,V.~Gorshkov}
\date{}

\maketitle

\makeatletter
\renewcommand{\@makefnmark}{}
\makeatother
Most of the known Fourier transforms associated with the equations of mathematical physics have a trivial kernel, and an inversion formula as well as the Parseval equality are fulfilled. In other words, the system of the eigenfunctions involved in the definition of the integral transform is complete.

However, in some cases, the differential operator, in addition to the continuous part of the spectrum that defines this transform, may contain a set of eigenfunctions $\{e_k\}$, and the Parseval's equality takes the form
\begin{eqnarray}\label{parseval}
\|f\|^2 = \|F[f]\|^2+\sum_k (f,e_k)^2,
\end{eqnarray}
where $\{e_k\}$ become the elements of the kernel of  $F$. In that case $F$ will be called as a {\it degenerate transform}.

For example, the sine-Fourier transform 
$$
\hat f(\lambda) =\sqrt \frac 2 \pi \int_0^\infty \sin (\lambda s) f(s) ds 
$$
is based on the eigen functions of $A=d^2/dx^2$ in $L_2(0,\infty)$ with the Dirichlet condition $f(0)=0$. The spectrum of the operator is continuous and fills the entire negative half-axis: $\sigma_c=(-\infty,0]$.  This transform is not degenerate, and the inversion formula has the form
$$
f(x) = \frac 2 \pi \int_0^\infty \sin (\lambda x) \left ( \int_0^\infty \sin (\lambda s) f(s) ds \right ) d\lambda.
$$

Functions $\varphi(x, \lambda) = \sin(\lambda x)$ are the {\it generalized eigenfunctions} (i.e. not belonging to the main space $L_2(0,\infty)$, where the operator is set):
$$\partial_{xx}\varphi(x, \lambda) = -\lambda^2 \varphi(x, \lambda).$$

Similarly, the same operator with the Neumann condition generates a co\-sine Fourier transform. However, this operator with a Robin boundary condition $f'(0) + a f(0) = 0$
with $a>0$, in addition to its generalized eigen functions  contains an ordinary eigenfunction $e^{-ax}$ with an eigenvalue of $a^2$. Its spectrum consists of a continuous part $\sigma_c=(-\infty,0]$ and an eigenvalue $\lambda=a^2$. The Fourier transform generated by this operator will already have a nontrivial kernel containing this function (see, for example, \cite{Tm1}). However, the mixed boundary condition with $a>0$ refers to {\it non-physical}, and, as a rule, is not considered in differential equations.

Another important example of a family of operators with degenerate transform are
$$
\Delta_k  = \frac 1r \frac {\partial}{\partial r}\left(r \frac {\partial}{\partial r}\right) - \frac{k^2}{r^2},~k\in \mathbb{R}.
$$
For integers $k$, these operators are the Fourier coefficients of the Laplacian when decomposed into a Fourier series by an angular variable in polar coordinates.

These operators with a mixed boundary condition
$$
r_0\frac{\partial w(t,r)}{\partial r}\Big|_{r=r_0} \pm k w(t,r_0) = 0
$$
in addition to the continuous spectrum $\sigma_c=(-\infty,0)$ also have an eigenvalue $\lambda=0$ from the kernel of $\Delta_k$ (here $r_0>0$ is a fixed number that can be treated as the radius of the circle). The eigenfunctions from the kernel are $1/r^{\pm k}$ (for certain values of $k$). They also produce a degenerate transformation, which is a generalization of the integral Weber transform, the properties of which were investigated by the author in \cite{G1}.

And these transforms, precisely as a degenerate ones, had found an application in mathematical physics. With its help, the author obtained a solution to the Stokes problem in the exterior of the circle (see \cite{G2}). All this confirms  the  importance of such a degenerate transforms not only for the theory of functions, but also as having significant practical importance.

\vskip 10pt

Let the Sturm-Liouville operator $Af(x) = f'' - q(x)f$, defined on the semiaxis $x\in(0,\infty)$, have a spectrum $\sigma(A)$, which consists of a continuous part $E\subset\mathbb{R}$, and no bigger than a countable set of eigenvalues $\{\lambda_k\}$ of the finite multiplicity  with eigenfunctions $\{e_k\}$. As is known(\cite{Tm1}\cite{L1}), if $q(x)$ is continuous on $\mathbb{R_+}$, then there exist the generalized eigenfunctions $\varphi(x,\lambda)$ specifying the Fourier transform
$$
F[f]=\int_0^\infty \varphi (x,\lambda)f(x)dx,~A\varphi (x,\lambda) = \lambda \varphi (x,\lambda).
$$
At the same time, there is a spectral function $\rho(\lambda)$ such that $F[f]$ belong to the space $L_2$ with the weight $\rho(\lambda)$: $F[f] \in L_2\left (E,\rho(\lambda)\right)$.

\begin{defn} Let's call the functions $\varphi(x,\lambda)$ to be   {\it orthonormal}, denoting as $\left (\varphi (\cdot,\lambda), \varphi (\cdot,\zeta)\right )$ $=$ $\delta(\lambda-\zeta)$ if
\begin{align} \label{ortcond}
\hat f(\lambda) = \int_0^\infty \varphi (x,\lambda) \left (
\int_E \varphi (x,\zeta) \hat f(\zeta) d\rho(\zeta) \right ) dx
\end{align}
for any function $\hat f(\lambda)$ of the form $\hat f(\lambda) = F[f(\cdot)](\lambda)$, $f(x)\in L_2(\mathbb{R_+})$.
\end{defn}
Equality in this definition follows from the formal applying of the $\delta$-function $\delta(\lambda-\zeta)$ to $\hat f(\lambda)$.

\begin{defn}We will say that $\left\{ \{\varphi (\cdot,\lambda)\} ,\{e_k\}\right\}$ is a complete orthonormal system of eigenfunctions of the operator $A$ if $\left (\varphi (\cdot,\lambda), \varphi (\cdot,\zeta)\right ) = \delta(\lambda-\zeta)$,
$\left (\varphi (\cdot,\lambda), e_k\right )_{L_2(\mathbb{R_+})} = 0$,
$\left(e_k, e_j\right )_{L_2(\mathbb{R_+})} =\delta_{k,j}$ for any $\lambda, \zeta\in E$, any eigenfunctions $e_k, e_j$, and the Parseval equality (\ref{parseval}) holds.
\end{defn}

The condition $\left (\varphi (\cdot,\lambda), e_k\right )_{L_2(\mathbb{R_+})} = 0$ means that $F$ is the degenerate transform, i.e. $F[e_k]=0$.

\begin{thm}
Let the Sturm-Liouville operator $A$ be a generator of a strongly continuous  semigroup $e^{tA}$ in $L_2(0,\infty)$; its spectrum is real, bounded from above, consists of a continuous part and no bigger than a countable set of eigenvalues $\{e_k\}$, and the resolvent satisfies the estimate $\|R(A,\lambda)\|\leq C/\lambda$ with some $C>0$. Then the system of its eigen functions forms a complete orthonormal system and the following inversion formula holds:
\begin{align} \label{invformula}
f(x) =
\int_{E} \varphi(x,\lambda)\left ( \int_{\mathbb{R_+}} \varphi(s,\lambda) f(s)ds \right ) f(x)d \rho(\lambda) +
\sum_{k} (f,e_k)e_k.
\end{align}
\end{thm}

\begin{proof} Consider the boundary value problem in $L_2(\mathbb{R_+})$:
$$
\partial_t y(t,x) - A y=0, y(0,x)=f(x).
$$

Using the estimate on the resolvent $R(A,\lambda)$ from the conditions of the theorem, the solution $y(t,x)=e^{tA}$ of this equation can be given by the formula
\begin{align*}
y(t,x) = \frac 1{2\pi i}\int_\gamma e^{\lambda t}R(A,\lambda)f(x)d\lambda,
\end{align*}
where the contour $\gamma$ covers the real spectrum of the operator $A$. In the case of an unbounded spectrum $\gamma$ is the boundary of the sector $S_{a,\theta}=\{\lambda\in \mathbb C,~|arg(\lambda -a)|>\theta\}$ with some $a>0$, $\theta\in (\frac\pi 2, \pi)$.

The resolvent has a gap along the set $E$, and the points $\lambda_k$ are its poles. Integral
$$
P_{k} = \frac 1{2\pi i}\int_{|\lambda - \lambda_k|=\varepsilon} R(A,\lambda)d\lambda
$$
for a sufficiently small $\varepsilon$ is a finite-dimensional projector onto the proper subspace $E_k$ (see \cite{Kato}). And then the solution takes the following form:
\begin{align*}
y(t,x) =
\frac 1{2\pi i}\int_{E} e^{\lambda t}\left ( R(A,\lambda -i\cdot 0) - R(A,\lambda + i\cdot 0) \right ) f(x)d\lambda +
\sum_{k} e^{\lambda_k t} (f,e_k)e_k.
\end{align*}

The integral in the last equality defines a family of projectors on the generalized proper subspace \cite{Tm1} and can be expressed in terms of the eigenfunctions $\varphi(x,\lambda)$, giving the final formula for $y(t,x)$:
\begin{align}\label{eqsol}
y(t,x) =
\int_{E} e^{\lambda t} \varphi(x,\lambda)\left ( \int_{\mathbb{R_+}} \varphi(s,\lambda) f(s)ds \right ) f(x)d \rho(\lambda) +
\sum_{k} e^{\lambda_k t} (f,e_k)e_k.
\end{align}

Let's prove the inversion formula (\ref{invformula}). Then it will follow the completeness of the eigen functions. Since $A$ is the generator of the strongly continuous semigroup $e^{tA}$, then $y(t,\cdot)$ $\to f(x)$ strongly at $t\to 0$. If the set of the continuous spectrum $E$ is bounded, then passing to the limit at $t \to 0$ is allowed under the sign of the integral in (\ref{eqsol}) and the inversion formula is proved. If $E$ is not bounded, then we must justify the limit transition at $t\to 0$. It is enough to prove the validity of a limit transition only in the integral part of the equality (\ref{eqsol}). Let's prove weak convergence at $t\to 0$:
\begin{align*}
y(t,x) \rightharpoondown
\int_{E} \varphi(x,\lambda)\left ( \int_{\mathbb{R_+}} \varphi(s,\lambda) f(s)ds \right ) f(x)d \rho(\lambda) +
\sum_{k} (f,e_k)e_k.
\end{align*}

Take an arbitrary $g\in C_0^\infty(r_0,\infty)$. Then
$\hat g(\lambda)=F[g(\cdot)](\lambda)$ will decrease rapidly by $\lambda$. The following is true

\begin{prop}For an arbitrary $g\in C_0^\infty(r_0,\infty)$ holds $\lambda\hat g(\lambda)\in L_2\left (E,\rho(\lambda)\right)$.
\end{prop}

\begin{proof}
Since $A\varphi = \lambda\varphi$, then
\begin{align*}
\hat g(\lambda)=\int_0^\infty \varphi (x,\lambda)g(x)dx=
\frac1{\lambda}\int_0^\infty A[\varphi (x,\lambda)]g(x)dx=
\frac1{\lambda}\int_0^\infty \varphi (x,\lambda) A[g(x)]dx.
\end{align*}
Because $A[g(x)]\in L_2(\mathbb{R_+})$, then
$\int_0^\infty \varphi (x,\lambda) A[g(x)]dx \in L_2\left (E,\rho(\lambda)\right)$ and the proposition is proved.
\end{proof}

Denote $F^*[\hat g(\cdot)](x) = \int_{E} \varphi(x,\lambda) \hat g(\lambda)d\rho(\lambda)$. Then

\begin{align*}
\left (F^* \left [ e^{\lambda t} F[f] \right ], g(\cdot) \right )_{L_2(\mathbb{R_+})} =
\left ( e^{\lambda t}F[f], F[g] \right )_{L_2\left (E,\rho(\lambda)\right)}=
\left ( e^{\lambda t}\hat f(\lambda), \hat g (\lambda) \right )_{L_2\left (E,\rho(\lambda)\right)},
\end{align*}
where $\hat f(\lambda) = F[f](\lambda)$,
$\hat g(\lambda) = F[g] (\lambda)$.

The residuals of integrals
\begin{align*}
\int_{-\infty}^L \left | e^{\lambda t} \hat f(\lambda) \hat g (\lambda) \right | d\rho(\lambda) \leq \int_{-\infty}^L \left | \hat f(\lambda) \hat g (\lambda) \right | d\rho(\lambda)
 \leq \|\hat f(\lambda)\|_{L_2\left (-\infty, L \right )} \| \hat g (\lambda) \lambda \|_{L_2\left (-\infty,L \right )}
\end{align*}
converge to zero uniformly over $t$ at $L\to - \infty$. Consequently, we have proved the validity of the transition at $t\to 0$:
$$
\left ( e^{\lambda t}\hat f(\lambda), \hat g (\lambda) \right )_{L_2\left (E,\rho(\lambda)\right)} \to \left (\hat f(\lambda), \hat g (\lambda) \right )_{L_2\left (E,\rho(\lambda)\right)}.
$$

From the uniqueness of the weak limit, taking into account $y(t,\cdot)\rightarrow f(\cdot)$, the inversion formula is valid almost everywhere
\begin{align} \label{invformula_f}
f(\cdot)=F^*\left [F [f] \right ]+\sum_{k} (f,e_k)e_k,
\end{align}
which is the same formula as (\ref{invformula}).

Let's prove that ${e_k} \in Ker(F)$:
\begin{align*}
\int_0^\infty \varphi (x,\lambda)e_k(x)dx=\frac1{\lambda}
\int_0^\infty A[\varphi (x,\lambda)]e_k(x)dx=\frac1{\lambda}
\int_0^\infty \varphi (x,\lambda)Ae_k(x)dx\\=
\frac{\lambda_k}{\lambda}
\int_0^\infty \varphi (x,\lambda) e_k(x)dx,
\end{align*}
which implies $F[e_k]=0$.  
  
The orthogonality condition (\ref{ortcond}) follows from the inversion formula (\ref{invformula_f}) if we apply the transform $F$ to the latter:
$$
\hat f = F\left [ F^* \hat f \right ].
$$
The Parseval equality (\ref{parseval}) follows from the inversion formula (\ref{invformula}) in virtue of the ortho\-nor\-ma\-li\-ty of $\left\{ \{\varphi (\cdot,\lambda)\} ,\{e_k\}\right\}$. The theorem is  proved.  
\end{proof}

{\bf A.\,V.~Gorshkov}

MSU Faculty of Mechanics and Mathematics.

{\it E-mail}: alexey.gorshkov.msu@gmail.com

\end{document}